\documentclass[12pt]{amsart}
\usepackage{graphics,latexsym,amssymb,amsmath}
\usepackage[dvips]{graphicx}
\usepackage{psfrag}
\input{amssym.def}

\newenvironment{Proof}{\noindent{\bf Proof:}}{\hfill$\square$}
\newtheorem{thm}{Theorem}[section]

\newtheorem{lemma}[thm]{Lemma}
\newtheorem{rmk}[thm]{Remark}
\newtheorem{cor}[thm]{Corollary}
\newtheorem{prop}[thm]{Proposition}

\newtheorem{conj}{Conjecture}

\newcommand{\EPf}{\hbox{}\hfill$\Box$\vspace{.5cm}}

\newcommand{\Z}{{{\mathbb Z}}}

\newcommand{\F}{{{\mathbb F}}}
\newcommand{\T}{{{\mathcal T}}}
\newcommand{\Ss}{{{\mathcal S}}}
\newcommand{\Bz}{{{\mathbf 0}}}
\newcommand{\G}{{{\mathcal G}}}

\date{}
\begin{document}
\title{Cayley graph expanders and groups of finite width}
\author[N.~Peyerimhoff]{Norbert Peyerimhoff}
\author[A.~Vdovina]{Alina Vdovina}

\address[N.~Peyerimhoff]{Department of Mathematical Sciences, Durham
  University, Science Laboratories South Road, Durham, DH1 3LE, UK}
\email{norbert.peyerimhoff@durham.ac.uk}
\urladdr{www.maths.dur.ac.uk/~dma0np/}

\address[A.~Vdovina]{School of Mathematics and Statistics, Newcastle
  University, Newcastle-upon-Tyne, NE1 7RU, UK}
\email{alina.vdovina@ncl.ac.uk}
\urladdr{www.ncl.ac.uk/math/staff/profile/alina.vdovina}

\begin{abstract}
  We present new infinite families of expander graphs of vertex degree
  $4$, which is the minimal possible degree for Cayley graph
  expanders. Our first family defines a tower of coverings (with
  covering indices equals $2$) and our second family is given as
  Cayley graphs of finite groups with very short presentations with
  only $2$ generators and $4$ relations. Both families are based on
  particular finite quotients of a group $G$ of infinite upper
  triangular matrices over the ring ${\rm M}(3,\F_2)$.

  We present explicit vector space bases for the finite abelian
  quotients of the lower exponent-$2$ groups of $G$ by upper
  triangular subgroups and prove a particular $3$-periodicity of
  these quotients. 

  The pro-$2$ completion of the group $G$ satisfies the
  Golod-Shafa\-revich inequality
  $$ |R| \ge \frac{|X|^2}{4}, $$
  it is infinite, not $p$-adic analytic, contains a free nonabelian
  subgroup, but not a free pro-$p$ group.  We also conjecture that the
  group $G$ has finite width $3$ and finite average width $8/3$.
\end{abstract}

\maketitle

\section{Introduction}

The first explicit construction of expander graphs was introduced by
Margulis \cite{Marg1} and was an application of Kazhdan's property
(T). Expanders are simultaneously sparse and highly connected and are
not only of theoretical importance but also useful in computer
science, e.g., for network designs. An extensive survey of this topic
is given in \cite{HLW}.

Cartwright and Steger \cite{CS} constructed infinite sequences of
groups acting simply transitively on the vertices of buildings of type
$\tilde A_n$. Using these groups, Lubotzky, Samuels and Vishne
\cite{LSV} and, independently, Sarveniazi \cite{Sar} presented
explicit constructions of expanders and proved that they are, in fact,
(higher dimensional) Ramanujan complexes.

M. Ershov constructed in \cite{Er} for every sufficiently large prime
$p$ a group violating the Golod-Shafarevich inequality (as
presented in \cite[p. 87]{LS}) and having property (T). Since
those groups have infinitely many $p$-quotients, this fact
provided new families of expanders in connection with pro-$p$ groups.

In this article, we bring many of these aspects together. In Section
\ref{expanders}, we present two new families of Cayley graph
  expanders of vertex degree $4$. One family is given by a tower of
coverings with covering indices equals $2$, and the other family is
given by very short presentations with {\em $2$ generators} and {\em
  only $4$ relations}:

\begin{thm} \label{expshort}
  The groups
  $$ 
  G_k := \langle\, x_0,x_1\, |\, r_1, r_2, r_3, [x_1,\ {}_k\ x_0] \, \rangle, 
  $$ 
  with
  \begin{align*}
    r_1 &= x_1 x_0 x_1 x_0 x_1 x_0 x_1^{-3} x_0^{-3}, \\
    r_2 &= x_1 x_0^{-1} x_1^{-1} x_0^{-3} x_1^2 x_0^{-1} x_1 x_0 x_1,\\
    r_3 &= x_1^3 x_0^{-1} x_1 x_0 x_1 x_0^2 x_1^2 x_0 x_1 x_0,
  \end{align*}
  are finite and the associated Cayley graphs with respect to the
  symmetric set $\{ x_0^{\pm 1}, x_1^{\pm 1} \}$ define an infinite
  family of expanders of vertex degree $4$, satisfying $| G_i | \to
  \infty$.
\end{thm}

Each expander graph has twice as many edges as it has vertices. To
prove that our graphs are expanders indeed, we present them as Cayley
graphs of finite quotients of a group $G$ acting cocompactly on an
Euclidean building of type $\tilde A_2$, like in \cite{LSV} or
\cite{Sar}, but our graphs are not Ramanujan for sufficiently many vertices.

The group $G$ is given by $\langle x_0,x_1 \mid r_1, r_2, r_3
\rangle$. The pro-$2$ completion $\widehat G_2$ of $G$ can be
considered as a finitely presented pro-$2$ group. In Section
\ref{outlook}, we discuss particular properties of this pro-$2$ group.
Even though $\widehat G_2$ satisfies the Golod-Shafarevich inequality,
it is still infinite. Lubotzky \cite{Lub1} has shown that $p$-adic
analytic groups satisfy the Golod-Shafarevich inequality. The group
$\widehat G_2$, however, is not $p$-adic analytic. Wilson \cite{Wi}
conjectured that discrete, resp., pro-$p$ groups violating the
Golod-Shafarevich inequality have free subgroups, resp., free
pro-$p$ subgroups of rank two. The first conjecture was later proved
in \cite{WZ} and the proof of the second conjecture can be found in
\cite[p. 224]{Zel}. The group $\widehat G_2$ satisfies this
inequality, doesn't contain a free pro-$2$ subgroup, but it
nevertheless contains a free subgroup of rank two.

Our considerations are based on a linear representation of the group
$G$ by infinite upper unitriangular matrices over the field
$\F_2$. It is particularly useful to view these infinite matrices as
been built up by diagonals of $3 \times 3$ block matrices. Natural
normal subgroups $H_i$ are given by infinite upper triangular matrices
with vanishing first $i$ upper diagonals. Let 
$$ G = \lambda_0(G) \ge \lambda_1(G) \ge \cdots $$
denote the lower exponent-$2$ series of $G$. Obviously, we have
$\lambda_i(G) \le G \cap H_i$. In Theorem
\ref{gammabases} of Section \ref{main}, we prove a particular
{\em $3$-periodicity} for the abelian quotients $\lambda_i(G) /
(\lambda_i(G) \cap H_{i+1})$ and derive {\em explicit bases} for them. These
considerations give useful information about the structure of our
families of Cayley graph expanders.

Computer calculations show for the group $G$ the identities
$\lambda_i(G) = G \cap H_i$ from $i \ge 1$ onwards and $\lambda_i(G) /
\lambda_{i+1}(G) \cong \gamma_i(G) / \gamma_{i+1}(G)$ from $i \ge 2$
onwards, up to the index $i = 100$. Here, $\gamma_i(G)$ denote the
lower central series groups of $G$. If these identities are true for
all $i$-indices, then our group $G$ has {\em finite width $3$} and
finite average width $8/3$, and the {\em covering indices} of our
expander graphs $\G_i$ are given by the $3$-periodic sequence $4,8,\,
\overline{4,8,8}$.

The group $G$ is a subgroup of a group $\Gamma$ belonging to a class
of groups $\Gamma_\T$, which were constructed in \cite[Section]{CMSZ}
and are related to particular triangle presentations $\T$ of special
finite projective planes of prime power order $q$. We expect that
analogous {\em finite width properties} hold also for these groups
$\Gamma_\T$ (see Conjecture \ref{conjgt} in Section
\ref{outlook}). Again, this conjecture has been checked by computer
for many $i$-indices for the prime powers $q=2,4,5,7,9,11$.

\medskip

{\bf Acknowledgement:} The authors are grateful to M. Belolipetsky,
A. Borovik, M. du Sautoy, R. Grigorchuk, M. Kontsevich,
Ch. R. Leedham-Green, A. Panchishkin, M. Sapir and P. Zalesskii for many helpful
discussions.

\section{Commutator schemes of the group $G$}
\label{main}

Our group $G$ is a fundamental group of a simplicial complex $\mathcal K$,
consisting of $14$ triangles, and defined by the labeling scheme in
Figure \ref{labelscheme}. Straightforward calculations give the following
presentation of this group:
\begin{align} \label{r1r2r3}
r_1 &= x_1 x_0 x_1 x_0 x_1 x_0 x_1^{-3} x_0^{-3},\nonumber \\
r_2 &= x_1 x_0^{-1} x_1^{-1} x_0^{-3} x_1^2 x_0^{-1} x_1 x_0 x_1,\\
r_3 &= x_1^3 x_0^{-1} x_1 x_0 x_1 x_0^2 x_1^2 x_0 x_1 x_0.\nonumber 
\end{align}

\begin{figure}[h]
  \begin{center}      
      \psfrag{x0}{$x_0$}
      \psfrag{x1}{$x_1$}
      \psfrag{x2}{$x_2$}
      \psfrag{x3}{$x_3$}
      \psfrag{x4}{$x_4$}
      \psfrag{x5}{$x_5$}
      \psfrag{x6}{$x_6$}
      \psfrag{y0}{$y_0$}
      \psfrag{y1}{$y_1$}
      \psfrag{y2}{$y_2$}
      \psfrag{y3}{$y_3$}
      \psfrag{y4}{$y_4$}
      \psfrag{y5}{$y_5$}
      \psfrag{y6}{$y_6$}
    \includegraphics[height=3.8cm]{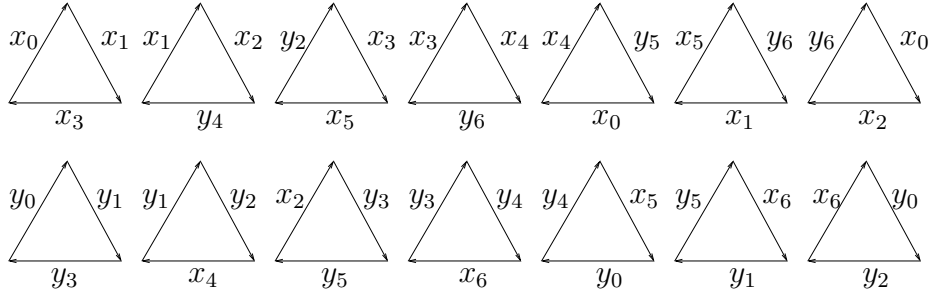}
  \end{center}
   \caption{Labeling scheme for our simplicial complex}
   \label{labelscheme}
\end{figure}

This complex is a $2$-fold cover of a classifying space of the group 
\begin{equation} \label{G2} 
  \Gamma := \langle x_0,\dots,x_6 | x_i
  x_{i+1} x_{i+3}\ {\mathrm{for}}\ i = 0,1,\dots,6 \rangle,
\end{equation}
where $i,i+1,i+3$ are taken mod $7$. The group $\Gamma$ belongs to the
family of groups introduced in \cite[Section 4]{CMSZ}.

\begin{prop}
  The group $\Gamma$ in \eqref{G2} is generated by $x_0,x_1,x_2$ and
  the subgroup $G$ generated by $x_0,x_1$ is an index two normal
  subgroup of $\Gamma$.
\end{prop}

Even though $[\Gamma : G] = 2$ follows from covering arguments, we give
a different proof for this fact as well.

\smallskip

\begin{Proof}
  The relations imply that $x_3 = (x_0 x_1)^{-1}$, $x_4 = (x_1
  x_2)^{-1}$, $x_5 = x_0 x_1 x_2^{-1}$, $x_6 = (x_0 x_2)^{-1}$, so
  $\Gamma$ is generated by $x_0, x_1, x_2$. The diagrams in Figure
  \ref{diagrel} show the validity of the three relations $x_2 x_1 x_2
  = x_0^{-1} x_1^{-1} x_0^{-1}$, $x_2 x_0^{-1} x_2 = x_1{-1} x_0 x_1$
  and ${x_2}^2=x_0^{-1} x_1 x_0 x_1$ in $\Gamma$.

  \begin{figure}[h]
    \begin{center}
      \psfrag{0}{$x_0$}
      \psfrag{1}{$x_1$}
      \psfrag{2}{$x_2$}
      \psfrag{3}{$x_3$}
      \psfrag{4}{$x_4$}
      \psfrag{5}{$x_5$}
      \psfrag{6}{$x_6$}
      \includegraphics[height=4cm]{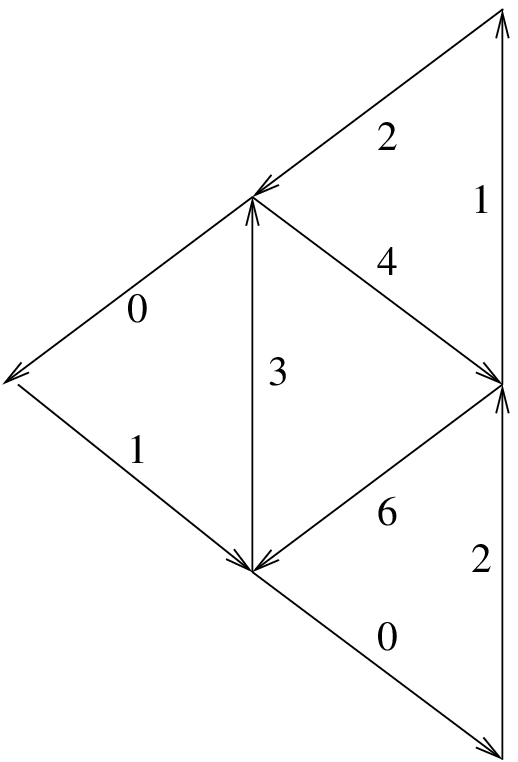}\hspace{2cm}
      \includegraphics[height=3cm]{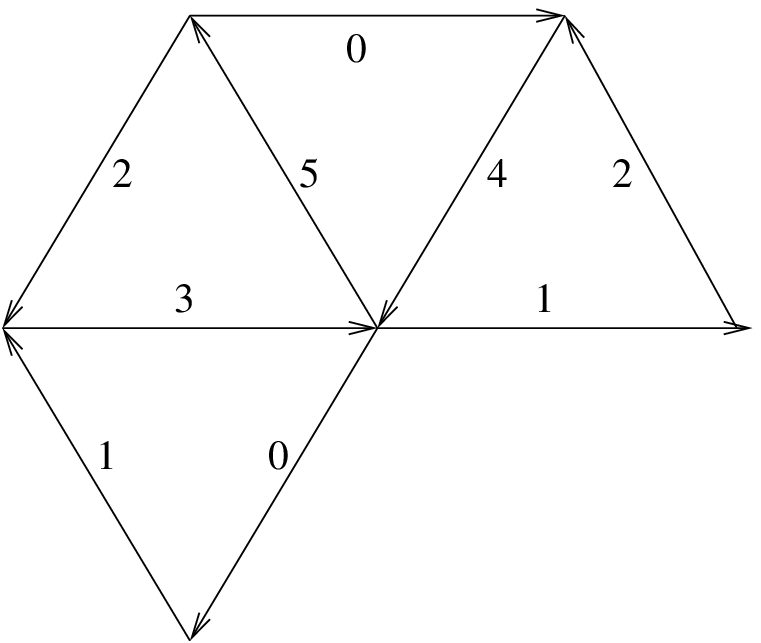}\\[.5cm]
      \includegraphics[height=1.5cm]{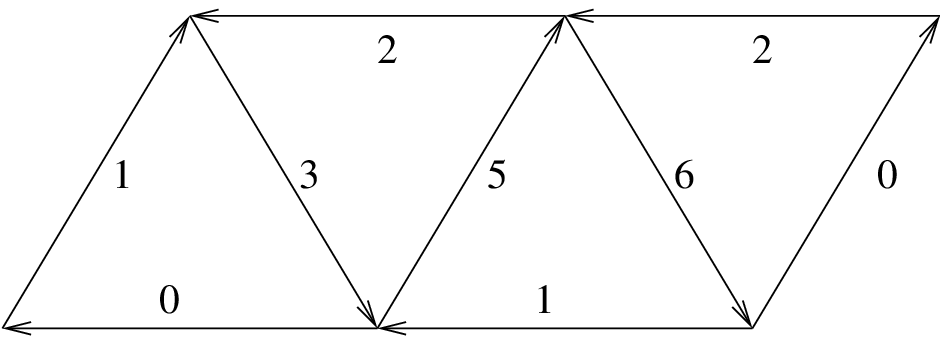}
   \end{center}
   \caption{Diagram for relations in $G_2$}
   \label{diagrel}
  \end{figure}

  It remains to prove that every element of $\Gamma$ can be written as
  $w$ or $w x_2$, where $w$ is a word in $x_0, x_1$. By relation
  ${x_2}^2=x_0^{-1} x_1 x_0 x_1$, it suffices to prove that any
  reduced word $x_2 w$ with $w$ only containing $x_0, x_1$ and being
  of length $n$ can be rewritten as $w_1 x_2 w_2$ with $w_1, w_2$ only
  containing $x_0, x_1$ and $w_2$ being of length $\le n-1$. Assume
  that $w = x_0 w'$. (The other cases $w = x_0^{-1} w'$, $w = x_1 w'$
  and $w = x_1^{-1} w'$ are treated similarly.) Then we have, using
  the above three relations:
  \begin{eqnarray*}
    x_2 x_0 w' &=& x_2^2 (x_2^{-1} x_0 x_2^{-1}) x_2 w' \\
    &=& x_2^2 (x_1^{-1} x_0^{-1} x_1) x_2 w' \\
    &=& (x_0^{-1} x_1 x_0 x_1) x_1^{-1} x_0^{-1} x_1 x_2 w' \\
    &=& x_0^{-1} x_1^2 x_2 w' = w_1 x_2 w_2,
  \end{eqnarray*}
  with $w_1 = x_0^{-1} x_1^2$ and $w_2 = w'$.
\end{Proof}

\medskip

Now, we employ a particular linear representation of $\Gamma$ in the
matrix group ${\rm GL}(9,\F_2(y))$ given in \cite{LSV} (note that
the $b_i$ in \cite[Section 10]{LSV} correspond to our $x_i^{-1}$):
{\small
\begin{eqnarray*}
x_0 &=& \begin{pmatrix} 1 & 0 & 0 & 0 & 0 & 0 & 0 & 0 & 0 \\
                        0 & 1 & 0 & 0 & 0 & 1 & 0 & 1 & 0 \\
                        0 & 0 & 1 & 0 & 1 & 1 & 0 & 0 & 1 \\
                        0 & 0 & 0 & 1 & 0 & 0 & 0 & 0 & 0 \\
                        0 & 0 & 0 & 0 & 1 & 0 & 0 & 0 & 1 \\
                        0 & 0 & 0 & 0 & 0 & 1 & 0 & 1 & 1 \\
                        0 & 0 & 0 & 0 & 0 & 0 & 1 & 0 & 0 \\
                        0 & 0 & 0 & 0 & 0 & 0 & 0 & 1 & 0 \\
                        0 & 0 & 0 & 0 & 0 & 0 & 0 & 0 & 1
        \end{pmatrix} + \frac{1}{y}
        \begin{pmatrix} 0 & 0 & 0 & 0 & 0 & 0 & 0 & 0 & 0 \\
                        0 & 1 & 1 & 0 & 0 & 1 & 0 & 1 & 0 \\
                        0 & 1 & 0 & 0 & 1 & 1 & 0 & 0 & 1 \\
                        0 & 0 & 0 & 0 & 0 & 0 & 0 & 0 & 0 \\
                        0 & 1 & 0 & 0 & 1 & 1 & 0 & 0 & 1 \\
                        0 & 0 & 1 & 0 & 1 & 0 & 0 & 1 & 1 \\
                        0 & 0 & 0 & 0 & 0 & 0 & 0 & 0 & 0 \\
                        0 & 0 & 1 & 0 & 1 & 0 & 0 & 1 & 1 \\
                        0 & 1 & 1 & 0 & 0 & 1 & 0 & 1 & 0
        \end{pmatrix},\\[.2cm]
x_1 &=& \begin{pmatrix} 1 & 0 & 0 & 0 & 0 & 0 & 0 & 0 & 0 \\
                        0 & 1 & 0 & 0 & 1 & 0 & 1 & 1 & 1 \\
                        0 & 0 & 1 & 1 & 1 & 1 & 0 & 1 & 1 \\
                        0 & 0 & 0 & 1 & 0 & 0 & 0 & 1 & 1 \\
                        0 & 0 & 0 & 0 & 1 & 0 & 1 & 0 & 0 \\
                        0 & 0 & 0 & 0 & 0 & 1 & 0 & 0 & 0 \\
                        0 & 0 & 0 & 0 & 0 & 0 & 1 & 0 & 0 \\
                        0 & 0 & 0 & 0 & 0 & 0 & 0 & 1 & 0 \\
                        0 & 0 & 0 & 0 & 0 & 0 & 0 & 0 & 1
        \end{pmatrix} + \frac{1}{y}
        \begin{pmatrix} 0 & 0 & 0 & 0 & 0 & 0 & 0 & 0 & 0 \\
                        0 & 1 & 1 & 0 & 1 & 0 & 1 & 1 & 1 \\
                        0 & 1 & 0 & 1 & 1 & 1 & 0 & 1 & 1 \\
                        0 & 1 & 0 & 1 & 1 & 1 & 0 & 1 & 1 \\
                        0 & 0 & 1 & 1 & 0 & 1 & 1 & 0 & 0 \\
                        0 & 0 & 0 & 0 & 0 & 0 & 0 & 0 & 0 \\
                        0 & 1 & 0 & 1 & 1 & 1 & 0 & 1 & 1 \\
                        0 & 0 & 1 & 1 & 0 & 1 & 1 & 0 & 0 \\
                        0 & 1 & 0 & 1 & 1 & 1 & 0 & 1 & 1
        \end{pmatrix}.
\end{eqnarray*}
}
We have $x_0 = A_0 + \frac{1}{y} A_1$ and $x_1 = B_0 + \frac{1}{y}
B_1$ with $9 \times 9$ matrices $A_0,A_1,B_0,B_1 \in {\rm M}(9,\F_2)$, and
their inverses $x_0^{-1}, x_1^{-1}$ are of the same form. Therefore,
an arbitrary group element $x \in G$ is of the form
$$ x = C_0 + \sum_{j=1}^k \frac{1}{y^j} C_j, $$
which we identify with the (finite band) upper triangular infinite
Toeplitz matrix
\begin{equation} \label{xrep}
x = \begin{pmatrix} C_0 & C_1 & C_2 & \dots & C_k & 0 & 0 & \dots \\
  0 & C_0 & C_1 & \dots & C_{k-1} & C_k & 0 & \ddots \\ 0 & 0 & C_0 &
  \dots & C_{k-2} & C_{k-1} & C_k & \ddots \\ \vdots & \ddots & \ddots
  & \ddots & \ddots & \ddots & \ddots & \ddots \end{pmatrix},
\end{equation}
where each $C_i$ is a matrix in ${\rm M}(9,\F_2)$. One checks that
multiplication of elements in ${\rm GL}(9,\F_2[1/y])$ and of the
corresponding infinite matrices is consistent.

For a more detailed analysis it is useful to rewrite the matrix
representation \eqref{xrep} of an arbitrary element $x \in G$ with the
help of $3 \times 3$ matrices. This requires some notation.

Let $\Ss$ denote the vector space of all $3 \times 9$ matrices over
$\F_2$, i.e, every $a \in \Ss$ is of the form
$$ a = \begin{pmatrix} a(1) & a(2) & a(3) \end{pmatrix} $$
with $3 \times 3$-matrices $a(i)$ over $\F_2$. Let $\Bz$ denote the
zero matrix in $\Ss$. For every (finite or infinite) sequence
$a_1,a_2,\dots \in \Ss$ let $M(a_1,a_2,\dots)$ denote the following
infinite matrix: All lower diagonals of size $3 \times 3$ are zero,
the main diagonal of size $3 \times 3$ consists only of identity
matrices and the $i$-th upper diagonal of size $3 \times 3$ has the
$3$-periodic entries $a_i(1),a_i(2),a_i(3),a_i(1),a_i(2),a_i(3),\dots$
for $i \ge 1$.  In the case of a finite sequence $a_1,\dots,a_p$, all
diagonals above the $p$-th diagonal of size $3 \times 3$ are also
chosen to be zero. The set of all such infinite matrices has an
obvious group structure and is denoted by $H$. For $i \ge 0$, let
$M_i(a_1,a_2,\dots)$ denote the matrices
$M(\underbrace{0,0,\dots,0}_{i\, \text{times}},a_1,a_2,\dots)$ and let
$H_i$ denote the normal subgroup of $H$ consisting of all those
matrices. Let 
$$ G = \lambda_0(G) \ge \lambda_1(G) \ge \lambda_2(G) \ge \cdots $$
be the lower exponent-$2$ series of $G$, i.e., $\lambda_{i+1}(G) =
[\lambda_i(G),\lambda_i(G)] \lambda_i(G)^2$ for $i \ge 0$. Note that
$\lambda_i(G) \le G \cap H_i$. With this notation, we have {\small
\begin{eqnarray}
  x_0 = M(a_1,\dots,a_5), \quad a_1 &=&
    \begin{pmatrix} 0 & 0 & 0 & 0 & 0 & 0 & 0 & 0 & 0 \\
                    0 & 0 & 1 & 0 & 0 & 1 & 0 & 0 & 1 \\
                    0 & 1 & 1 & 0 & 1 & 1 & 0 & 1 & 1 \end{pmatrix},
    \label{x0}
    \\[.2cm]
  x_1 = M(b_1,\dots,b_5), \quad b_1 &=&
    \begin{pmatrix} 0 & 0 & 0 & 0 & 1 & 1 & 0 & 1 & 0 \\
                    0 & 1 & 0 & 1 & 0 & 0 & 0 & 0 & 1 \\
                    1 & 1 & 1 & 0 & 0 & 0 & 0 & 1 & 0 \end{pmatrix}.
    \label{x1}
\end{eqnarray}
}

For our next result, we use the following shorthand notation for
higher commutators:
$$
[x,\ {}_n \ y,\, z] :=
[x,\underbrace{y,\dots,y}_{n},z].
$$

\begin{thm} \label{gammabases} 
  We have the following $3$-periodicity for the abelian quotients
  $\lambda_i(G)/(\lambda_i(G) \cap H_{i+1})$ for $i \ge 2$:
  $$
  [ \lambda_i(G) : (\lambda_i(G) \cap H_{i+1}) ] =
  \begin{cases} 8, & \text{if $i \equiv 0,1 \mod 3$,}\\ 4, & \text{if
      $i \equiv 2 \mod 3$,} \end{cases}
  $$
  and a basis of $\lambda_i(G) / (\lambda_i(G) \cap H_{i+1})$ (as
  vector space over $\F_2$) is given by
  \begin{eqnarray*}
    \{ [x_1,\ {}_i \ x_0], [x_1,\ {}_{i-2} \ x_0,x_1,x_0], [x_1,\ {}_{i-2} 
    \ x_0,x_1,x_1] \}, 
    & \text{if $i \equiv 1 \mod 3$,} \\
    \{ [x_1,\ {}_i \ x_0], [x_1,\ {}_{i-1} \ x_0,x_1] \}, 
    & \text{if $i \equiv 2 \mod 3$,} \\
    \{ [x_1,\ {}_i \ x_0], [x_1,\ {}_{i-1} \ x_0,x_1],
    [x_1,\ {}_{i-2} \ x_0,x_1,x_1] \}, & \text{if $i \equiv 0 \mod 3$,}
  \end{eqnarray*}
  where each commutator $[x_i,x_j,\dots ]$ above is an abbreviation
  for the left coset $[x_i,x_j,\dots ] (\lambda_i(G) \cap H_{i+1})$.
\end{thm}

\begin{rmk} \label{gammabases01}
  To complete the picture for $i=0,1$, we have
  $$
  [ G : (G \cap H_1) ] = 4 \quad \text{and} \ \ [ \lambda_1(G) :
  (\lambda_1(G) \cap H_2) ] = 8
  $$ 
  with basis $\{ x_0, x_1 \}$ in the first case and $\{ x_0^2, x_1^2,
  [x_1,x_0] \}$ in the second case. Note that the periodicity of the
  quotients $\lambda_i(G) / (\lambda_i(G) \cap H_{i+1})$ starts at
  $i=2$.
\end{rmk}

Before we start with the proof of Theorem \ref{gammabases}, let us state
an immediate consequence.

\begin{cor} \label{quotgroupest}
  We have
  $$
  [ G : (G \cap H_{3i+j}) ] \ge \begin{cases} 2^2, & \text{if $(i,j)
      = (0,1)$,}\\ 
  2^5, & \text{if $(i,j)=(0,2)$,}\\
  2^{8i-1+\mu(j)}, & \text{if $i \ge 1$ and $j \in \{0,1,2\}$,} \end{cases}
  $$
  where $\mu(0)= 0$, $\mu(1) = 3$ and $\mu(2) = 6$. 
\end{cor}

\begin{Proof}
  The estimates follow immediately from Theorem \ref{gammabases} and
  Remark \ref{gammabases01} via
  \begin{align*} 
  [ G &: (G \cap H_k) ]\\
    &= [ G : (G \cap H_1)] \cdot [ (G \cap H_1) : 
    (G \cap H_2) ] \cdots [ (G \cap H_{k-1}) : 
    (G \cap H_k) ] \\
    &\ge [ G : (G \cap H_1)] \cdot [ \lambda_1(G) : (\lambda_1(G) \cap H_2) ] 
    \cdots [ \lambda_{k-1}(G) \cap (\lambda_{k_1} \cap H_k) ]. 
\end{align*}
\end{Proof}

The rest of this section is devoted to the proof of Theorem \ref{gammabases}.

\begin{prop} \label{abc}
  We have
  \begin{equation} \label{Mipowers} 
  M_k(a,\dots)^2 = M_{2k+1}(c,\dots) 
  \end{equation}
  with $c \in \Ss$ satisfying $c(i) = a(i) a(i+k+1)$
  and
  \begin{equation} \label{Micomm} 
  [ M_k(a,\dots), M(b,\dots) ] = M_{k+1}(c,\dots), 
  \end{equation}
  with $c \in \Ss$ defined by $c(i) = a(i) b(i+k+1) - b(i)
  a(i+1)$ (where the indices $i,i+1,i+k+1$ are taken mod $3$).
\end{prop}

Since we are working in vector spaces over $\F_2$, there is no
difference between the expressions $\alpha + \beta$ and $\alpha -
\beta$, but we prefer to use minus signs such that the formulas would
also hold in vector spaces over other fields.

\begin{Proof}
  The identity \eqref{Mipowers} is easily checked. Equation
  \eqref{Micomm} requires considerably more work and is proved in the
  Appendix.
\end{Proof}

\medskip

Now, we introduce the following elements of $\Ss$:
{\small
\begin{align*}
& \alpha_1 =
   \begin{pmatrix} 0 & 0 & 0 & 0 & 1 & 1 & 0 & 1 & 0\\ 
                   0 & 1 & 0 & 1 & 0 & 0 & 0 & 0 & 1\\ 
                   1 & 1 & 1 & 0 & 0 & 0 & 0 & 1 & 0 \end{pmatrix},
& \beta_1 =
   \begin{pmatrix} 0 & 0 & 0 & 0 & 0 & 1 & 0 & 1 & 1\\ 
                   1 & 0 & 1 & 0 & 0 & 0 & 0 & 1 & 1\\ 
                   1 & 1 & 0 & 1 & 0 & 0 & 0 & 0 & 1\end{pmatrix},\\[.2cm]
& \gamma_1 = 
   \begin{pmatrix} 0 & 0 & 0 & 0 & 1 & 1 & 0 & 1 & 0\\ 
                   0 & 1 & 1 & 1 & 0 & 1 & 0 & 0 & 0\\ 
                   1 & 0 & 0 & 0 & 1 & 1 & 0 & 0 & 1\end{pmatrix},\\[.2cm]
& \alpha_2 =
   \begin{pmatrix} 0 & 0 & 0 & 0 & 1 & 0 & 0 & 0 & 1\\ 
                   0 & 0 & 1 & 0 & 1 & 0 & 1 & 0 & 0\\ 
                   1 & 1 & 0 & 0 & 1 & 1 & 1 & 0 & 0\end{pmatrix},
& \beta_2 =
   \begin{pmatrix} 0 & 0 & 0 & 0 & 0 & 0 & 0 & 0 & 0\\ 
                   0 & 1 & 1 & 0 & 1 & 1 & 0 & 1 & 1\\ 
                   0 & 1 & 0 & 0 & 1 & 0 & 0 & 1 & 0\end{pmatrix},\\[.2cm]
& \gamma_2 = 
   \begin{pmatrix} 0 & 0 & 0 & 0 & 1 & 1 & 0 & 1 & 0\\ 
                   1 & 0 & 0 & 0 & 1 & 0 & 1 & 1 & 1\\ 
                   1 & 1 & 1 & 0 & 0 & 0 & 0 & 1 & 0\end{pmatrix},\\[.2cm]
& \alpha_3 =
   \begin{pmatrix} 0 & 0 & 0 & 0 & 0 & 1 & 0 & 1 & 1\\ 
                   0 & 0 & 0 & 1 & 0 & 1 & 1 & 1 & 0\\ 
                   0 & 0 & 0 & 0 & 1 & 0 & 1 & 1 & 1\end{pmatrix},
& \beta_3 =
   \begin{pmatrix} 0 & 0 & 0 & 0 & 1 & 0 & 0 & 0 & 1\\ 
                   0 & 0 & 0 & 0 & 1 & 1 & 1 & 0 & 1\\ 
                   0 & 0 & 0 & 1 & 0 & 1 & 0 & 1 & 0\end{pmatrix}.
\end{align*}
}

\smallskip

Using Proposition \ref{abc} as well as $x_0 =
M(\alpha_1+\gamma_1,\dots)$ and $x_1 = M(\alpha_1,\dots)$, we obtain the
following commutator scheme by straightforward calculations:

\begin{prop} \label{commscheme}
  We have for every integer $k \ge 0$:
  {\small
  \begin{align*}
    &[M_{3k}(\alpha_1,\dots), x_0] = M_{3k+1}(\alpha_2,\dots),
    &[M_{3k}(\alpha_1,\dots), x_1] = M_{3k+1}(\Bz,\dots), \\
    &[M_{3k}(\beta_1,\dots), x_0] = M_{3k+1}(\beta_2+\gamma_2,\dots),
    &[M_{3k}(\beta_1,\dots), x_1] = M_{3k+1}(\beta_2,\dots), \\
    &[M_{3k}(\gamma_1,\dots), x_0] = M_{3k+1}(\alpha_2,\dots),
    &[M_{3k}(\gamma_1,\dots), x_1] = M_{3k+1}(\alpha_2,\dots),
  \end{align*}
  \begin{align*}
    &[M_{3k+1}(\alpha_2,\dots), x_0] = M_{3k+2}(\alpha_3,\dots),&
    &[M_{3k+1}(\alpha_2,\dots), x_1] = M_{3k+2}(\beta_3,\dots), \\
    &[M_{3k+1}(\beta_2,\dots), x_0] = M_{3k+2}(\Bz,\dots),&
    &[M_{3k+1}(\beta_2,\dots), x_1] = M_{3k+2}(\alpha_3,\dots), \\
    &[M_{3k+1}(\gamma_2,\dots), x_0] = M_{3k+2}(\beta_3,\dots),&
    &[M_{3k+1}(\gamma_2,\dots), x_1] = M_{3k+1}(\Bz,\dots),
  \end{align*}
  \begin{align*}
    &[M_{3k+2}(\alpha_3,\dots), x_0] = M_{3k+3}(\alpha_1,\dots),&
    &[M_{3k+2}(\alpha_3,\dots), x_1] = M_{3k+3}(\beta_1,\dots), \\
    &[M_{3k+2}(\beta_3,\dots), x_0] = M_{3k+3}(\beta_1,\dots),&
    &[M_{3k+2}(\beta_3,\dots), x_1] = M_{3k+3}(\gamma_1,\dots).
  \end{align*}
  }
\end{prop}

Note that this proposition can be used to calculate the $k$-th upper
diagonal $a \in \Ss$ of every $k$-fold commutator
$[x_{i_1},\dots,x_{i_k}] = M_k(a,\dots)$ with $i_1,\dots,i_k \in
\{0,1\}$. 

\medskip

To simplify notation, let $\Lambda_i := \lambda_i(G)$ and $L_{i+1} :=
\lambda_i(G) \cap H_{i+1}$ and $\Ss_1, \Ss_2, \Ss_3 \subset \Ss$ denote
the subspaces spanned by $\{ \alpha_1,\beta_1,\gamma_1 \}$, $\{
\alpha_2,\beta_2,\gamma_2 \}$ and $\{ \alpha_3,\beta_3 \}$,
respectively. Proposition \ref{abc} yields
\begin{equation} \label{x02x12x1x0}
x_0^2 = M_1(\beta_2,\dots), \quad x_1^2 = M_1(\gamma_2,\dots), \quad
\text{and}\ [x_1,x_0] = M_1(\alpha_2,\dots). 
\end{equation}
Since $M_1(a,\dots) \cdot M_1(b,\dots) = M_1(a+b,\dots)$ and
$M_1(a,\dots)^{-1} = M_1(-a,\dots)$, we conclude from
\eqref{x02x12x1x0} that $\Lambda_1 / L_2$ is spanned by $\alpha_2,
\beta_2, \gamma_2 \in \Ss$ (under the identification $a \mapsto M_1(a)
L_2$). $\alpha_2, \beta_2, \gamma_2$ are linear independent and,
therefore, $\Lambda_1 / L_2$ is $3$-dimensional and isomorphic to
$\Ss_2$.

Propositions \ref{abc} and \ref{commscheme} are the key ingredients
for the proof of Theorem \ref{gammabases}, which we carry out by
induction.

\bigskip

\noindent {\bf Proof of Theorem \ref{gammabases}:} We already know
that every element in $\Lambda_1$ is of the form $M_1(a,\dots)$ with
$a \in \Ss_2$. This is the begin of the induction.

\smallskip

Assume that we already know that every element in $\Lambda_{3k+1}$ is
of the form $M_{3k+1}(a,\dots)$ with $a \in \Ss_2$ for some $k \ge
0$. Using Propositions \ref{abc} and \ref{commscheme}, we conclude that
every element in $\Lambda_{3k+2}$ is of the form $M_{3k+2}(a,\dots)$
with $a \in \Ss_3$. Proposition \ref{commscheme} yields also that we
have
\begin{eqnarray*}
{}[x_1,\ {}_{3k-1} \ x_0] &=& M_{3k+2}(\alpha_3,\dots), \\
{}[x_1,\ {}_{3k-2} \ x_0,x_1] &=& M_{3k+2}(\beta_3,\dots).
\end{eqnarray*}
Since $\alpha_3,\beta_3$ span $\Ss_3$, we see that $\Lambda_{3k+2} /
L_{3k+3}$ is spanned by $[x_1,\ {}_{3k-1} \ x_0] L_{3k+3}$ and
$[x_1,\ {}_{3k-2} \ x_0,x_1] L_{3k+3}$. 

Repeating this reasoning twice, we obtain that $\Lambda_{3k+3}$ and
$\Lambda_{3k+4}$ contain only elements of the form $M_{3k+3}(a,\dots)$
and $M_{3k+4}(b,\dots)$ with $a \in \Ss_1$ and $b \in \Ss_2$, and that
the commutators given in the theorem are bases of the quotients
$\Lambda_{3k+3} / L_{3k+4}$ and $\Lambda_{3k+4} / L_{3k+5}$. This
completes the induction step $k \to k+1$ and thus the proof of the
theorem.\EPf

\section{Explicit construction of expanders}
\label{expanders}

The simplicial complex $\mathcal K$, introduced at the beginning of
Section \ref{main}, consists of $2$ vertices and $14$ triangular
faces. The link of each vertex is isomorphic to an incidence graph of
a finite projective plane of order $2$ (see \cite{V} for more details
of constructing polyhedra and determining their links). Using
\cite{BS} (see also \cite{Pa} or \cite{Z}) we conclude that the
fundamental group $G$ of $\mathcal K$ has property (T). We also
like to mention that the Kazdhan constants of the groups $\Gamma_\T$
of Desarguesian projective planes were {\em exactly} calculated in
\cite{CMS}.)

We choose the symmetric generating set $S := \{ x_0^{\pm 1}, x_1^{\pm
  1} \}$ of the group $G$. As explained, e.g., in
\cite[Prop. 3.3.1]{Lub2}, Kazdhan property (T) of $G$ implies that the
Cayley graphs of all quotients of finite index normal subgroups of $G$
(with respect to the set $S$) have a uniform positive lower bound on
their combinatorial Cheeger constants. Thus, any sequence of normal
subgroups with finite indices converging to infinity yields a family
of expanders. We choose the normal subgroups $N_i = G \cap H_i$. Note
that $[H : H_i]$ is a power of two, since the quotient $H/H_i$ can be
identified with the vector space $\Ss^i$ over $\F_2$ via the map
\begin{equation} \label{matident}
M(a_1,a_2,\dots) H_i \mapsto (a_1,a_2,\dots,a_i) \in \Ss^i. 
\end{equation}
This implies that the groups $N_i$ have finite indices in $G$ which
are, again, powers of $2$. We know from Corollary \ref{quotgroupest}
that these indices converge to infinity, so the corresponding Cayley
graphs $\G_i$ are expanders.

Let us now have a closer look at the explicit matrix models of the
quotients $G / N_i$, obtained via the identification
\eqref{matident}.  This identification induces a nonabelian group
structure on the space $\Ss^i$. In fact, the identity element is
$\Bz^i \in \Ss^i$ and we have
\begin{eqnarray*}
(a_1,a_2,\dots,a_i) \cdot (b_1,b_2,\dots,b_i) &=& (c_1,c_2,\dots,c_i),\\
(a_1,a_2,\dots,a_i)^{-1} &=& (d_1,d_2,\dots,d_i)
\end{eqnarray*}
with 
\begin{equation} \label{gmult}
c_j(k) = a_j(k) + b_j(k) + \sum_{s=1}^{j-1} a_s(k)b_{j-s}(k+s),
\end{equation}
where $k,k+s$ are taken mod $3$. For the coefficients $d_j(k)$, we 
obtain the recursion formulas $d_1(k) = -a_1(k)$ and
\begin{equation} \label{ginv}
d_j(k) = -a_j(k) - \sum_{s=1}^{j-1} a_s(k) d_{j-s}(k+s).
\end{equation}
The identification \eqref{matident} induces an embedding
$G / N_i \hookrightarrow \Ss^i$ and we denote the image of
the group $G / N_i$ in $\Ss^i$ by $K_i$. $K_i$ is generated by the images
of $x_0 N_i$ and $x_1 N_i$, which we denote by $v_0$ and $v_1$. Hence,
we have $v_0 = (a_1,\dots,a_5,\Bz,\dots) \in \Ss^i$ and $v_1 =
(b_1,\dots,b_5,\Bz,\dots) \in \Ss^i$ with $a_1, \dots, b_5$
defined in \eqref{x0} and \eqref{x1}. (If $ i < 5$, we set $v_0 =
(a_1,\dots,a_i)$ and $v_1 = (b_1,\dots,b_i)$.) The expanders $\G_i$ are the
Cayley graphs of $K_i$ with respect to $\{ v_0^{\pm 1}, v_1^{\pm 1} \}$.  
They are all regular graphs with vertex degree $4$ and 
\begin{equation} \label{tower}
\dots \G_i \to \G_{i-1} \to \dots \G_1 \to \G_0 
\end{equation}
is a tower of coverings. The covering indices of \eqref{tower} are
powers of $2$, since $[G : N_i]$ are powers of $2$. $\G_0$ and $\G_1$
are illustrated in Figure \ref{g01}.

\begin{figure}[h]      
      \psfrag{0}{{\small$\Bz$}}
      \psfrag{v0}{{\small$v_0$}}
      \psfrag{v1}{{\small$v_1$}}
      \psfrag{v0+v1}{{\small$v_0 v_1=v_1 v_0$}}
    \begin{center}
      \includegraphics[width=8cm]{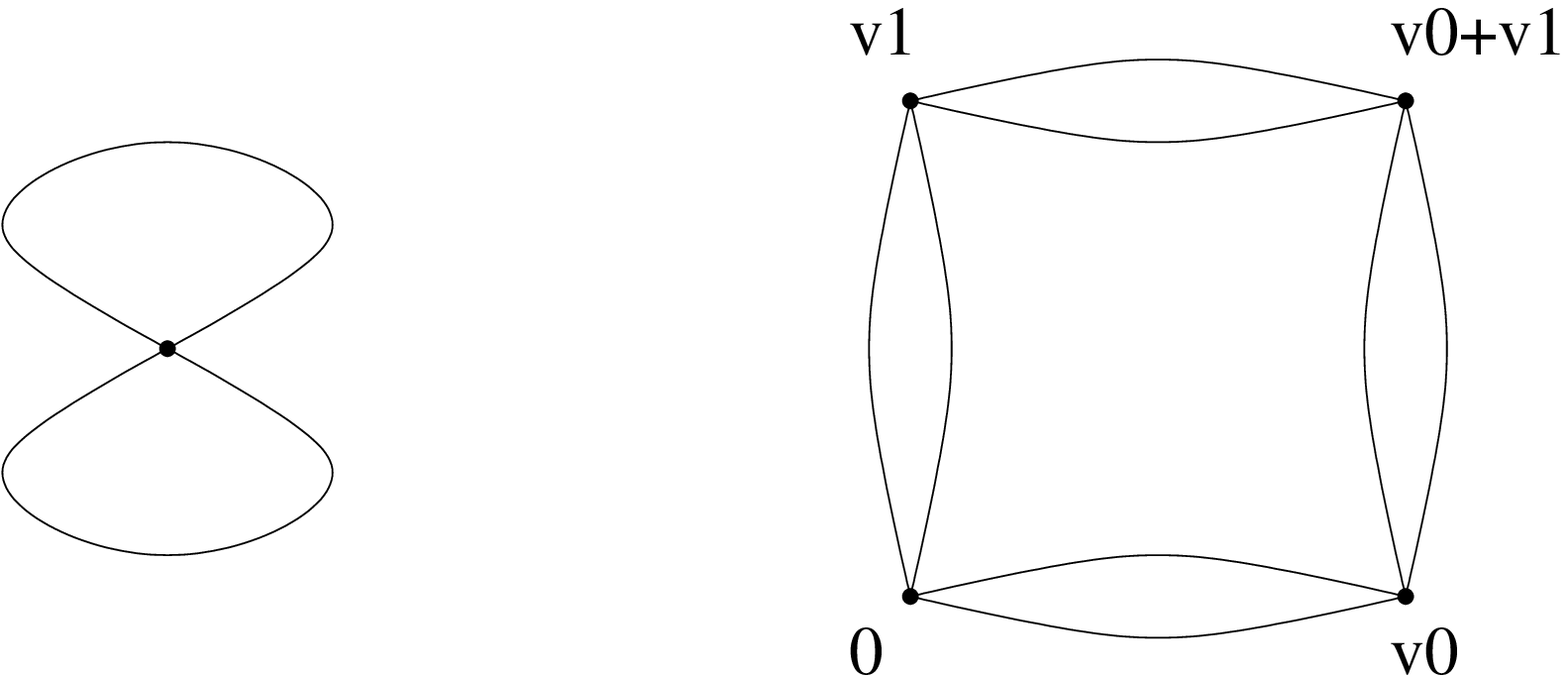}
    \end{center}
    \caption{The graphs $\G_0$ and $\G_1$}
    \label{g01}
\end{figure}

Let us briefly explain how to construct $\G_2$, a regular graph with
$2^5 = 32$ vertices. Let $v_0 = (a_1,a_2)$ and $v_1 = (b_1,b_2)$ be
the images of $x_0$ and $x_1$ in the group $K_2 \subset \Ss^2$ and
$w_1 = (\Bz,\alpha_2)$, $w_2 = (\Bz,\beta_2)$ and $w_3 =
(\Bz,\gamma_2)$. Note that $\Ss^2$ has both a vector space structure
and a nonabelian multiplicative group structure (given by
\eqref{gmult} and \eqref{ginv}). The elements of $K_2$ are given by
$$ K_2 := \bigoplus\limits_{v \in \{\Bz^2,v_0,v_1,v_0 \cdot v_1\}} v + \F_2 w_1 + 
\F_2 w_2 + \F_2 w_3 \subset \Ss^2, 
$$ 
and the center of $K_2$ is generated by $w_1, w_2, w_3$. Thus we have
$v \cdot w_i = w_i \cdot v = v + w_i$ for all $v \in K_2$. Using
$v_0^2=w_2, v_1^2 = w_3, [v_1,v_0] = w_1$ (which we compute with
\eqref{gmult} and \eqref{ginv} or we conclude it from
\eqref{x02x12x1x0}), we obtain
\begin{equation} \label{help} 
  v_1 \cdot v_0 = v_0 \cdot v_1 + w_1,
  \quad v_0^{-1} = v_0 + w_2, \quad \text{and}\ \ v_1^{-1} = v_1 +
  w_3.
\end{equation}
The vertices of $\G_2$ are the elements of $K_2$ and the
neighbours of a vertex $v \in \G_2$ are the vertices $v \cdot v_0^{\pm 1}$ and
$v \cdot v_1^{\pm 1}$.  These neighbours can all be calculated with
the help of \eqref{help}. The graph $\G_2$ is illustrated in Figure
\ref{g2} (where we use the abbreviations $w_{ij}$ and $w_{ijk}$ for
$w_i + w_j$ and $w_i + w_j + w_k$).

\begin{figure}[h]
      \psfrag{0}{{\small$\Bz^2$}}
      \psfrag{w1}{{\small$w_1$}}
      \psfrag{w2}{{\small$w_2$}}
      \psfrag{w3}{{\small$w_3$}}
      \psfrag{w12}{{\small$w_{12}$}}
      \psfrag{w13}{{\small$w_{13}$}}
      \psfrag{w23}{{\small$w_{23}$}}
      \psfrag{w123}{{\small$w_{123}$}}
      \psfrag{v0*v1}{{\small$v_0 v_1$}}
      \psfrag{v0*v1+w1}{{\small$v_0 v_1 +w_1$}}
      \psfrag{v0*v1+w2}{{\small$v_0 v_1 +w_2$}}
      \psfrag{v0*v1+w3}{{\small$v_0 v_1 +w_3$}}
      \psfrag{v0*v1+w12}{{\small$v_0 v_1 +w_{12}$}}
      \psfrag{v0*v1+w13}{{\small$v_0 v_1 +w_{13}$}}
      \psfrag{v0*v1+w23}{{\small$v_0 v_1 +w_{23}$}}
      \psfrag{v0*v1+w123}{{\small$v_0 v_1 +w_{123}$}}
      \psfrag{v0}{{\small$v_0$}}
      \psfrag{v0+w1}{{\small$v_0 +w_1$}}
      \psfrag{v0+w2}{{\small$v_0 +w_2$}}
      \psfrag{v0+w3}{{\small$v_0 +w_3$}}
      \psfrag{v0+w12}{{\small$v_0 +w_{12}$}}
      \psfrag{v0+w13}{{\small$v_0 +w_{13}$}}
      \psfrag{v0+w23}{{\small$v_0 +w_{23}$}}
      \psfrag{v0+w123}{{\small$v_0 +w_{123}$}}
      \psfrag{v1}{{\small$v_1$}}
      \psfrag{v1+w1}{{\small$v_1 +w_1$}}
      \psfrag{v1+w2}{{\small$v_1 +w_2$}}
      \psfrag{v1+w3}{{\small$v_1 +w_3$}}
      \psfrag{v1+w12}{{\small$v_1 +w_{12}$}}
      \psfrag{v1+w13}{{\small$v_1 +w_{13}$}}
      \psfrag{v1+w23}{{\small$v_1 +w_{23}$}}
      \psfrag{v1+w123}{{\small$v_1 +w_{123}$}}
    \begin{center}
      \includegraphics[width=12cm]{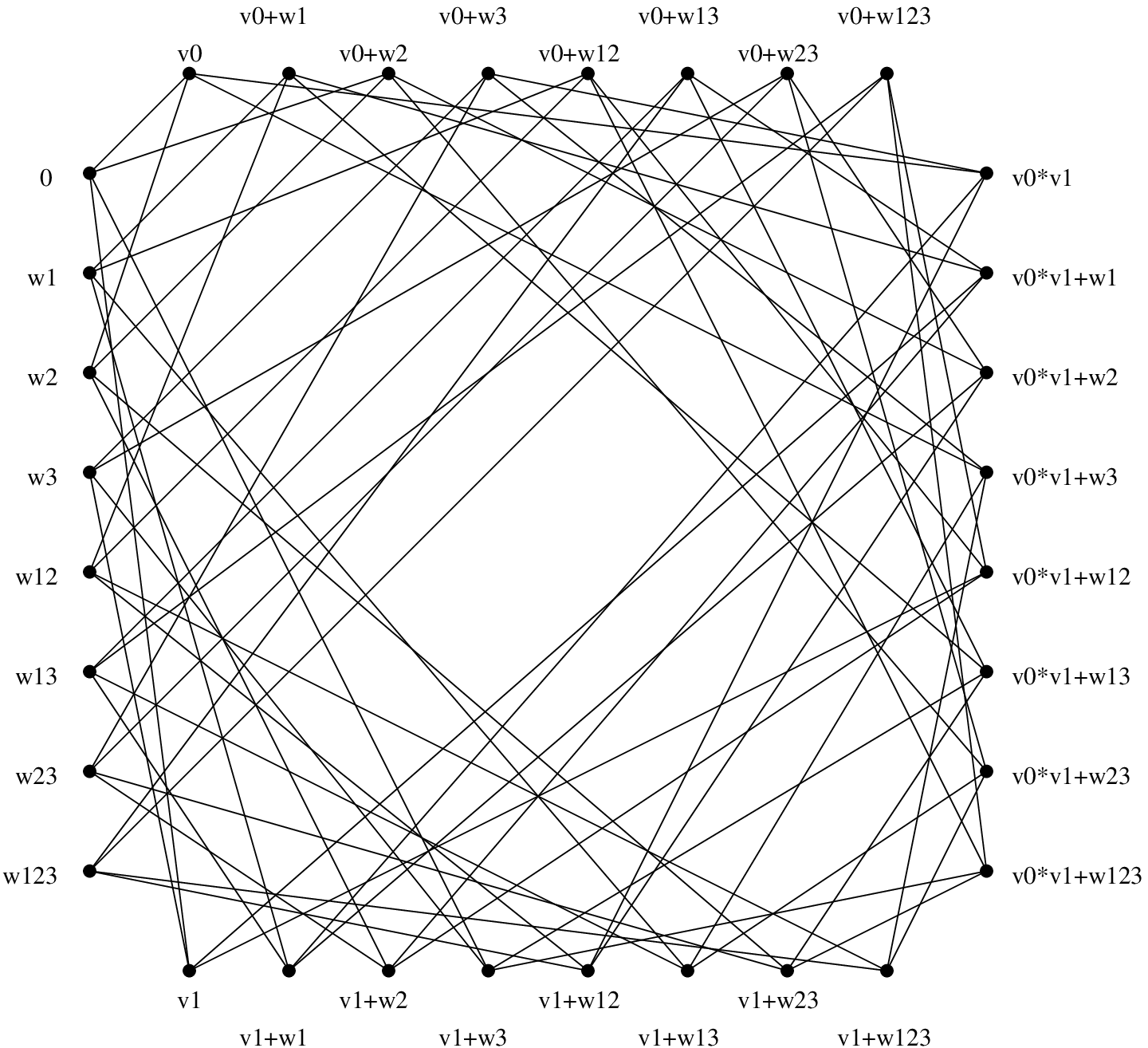}
    \end{center}
    \caption{The graph $\G_2$}
    \label{g2}
\end{figure}

Computer calculations with MAGMA show that the graph $\G_5$ with
$2^{13} = 8192$ vertices is the first graph which is {\em not
  Ramanujan}. Note also that we can fill into the tower of coverings
\eqref{tower} new intermediate covering graphs in order to obtain a
new tower of coverings
$$
\dots \widetilde \G_i \to \widetilde \G_{i-1} \to \dots \widetilde
\G_1 \to \widetilde \G_0,
$$ 
where the covering indices of two subsequent graphs {\em are exactly
  $2$}. This follows easily from the fact that every finite $2$-group
has a normal index $2$ subgroup. The graphs $\widetilde \G_i$ are
still expanders with the same positive lower bound on their
combinatorial Cheeger constants. This ``completed'' tower of $2$-fold
coverings fits well to results of Bilu and Linial \cite{BL}. They
present a construction of $2$-fold towers of covering graphs with
nearly optimal spectral gap. They also conjecture, based on extensive
numerical tests, that {\em every Ramanujan graph has a $2$-fold
  covering which is again Ramanujan}. If their conjecture is true,
there should be a different continuation of the sequence $\G_4 \to
\G_3 \to \dots \to \G_0$ by Ramanujan graphs.

The group $G$ can also be used to obtain another family of Cayley
graph expanders of minimal vertex degree $4$, presented in Theorem
\ref{expshort}, given by finite groups with two generators
and {\em only four relations}. 

\bigskip

\noindent {\bf Proof of Theorem \ref{expshort}:}
  We conclude from Proposition \ref{commscheme} that $[x_1,\ {}_k\
  x_0] = M_k(\alpha_{k+1},\dots)$ (taking $k$ in the index of
  $\alpha_{k=1}$ mod $3$), and hence this commutator never represents
  the identity in $G$. By the normal subgroup theorem (see, e.g.,
  \cite{Marg1}), we know that the group $G$ is just
  infinite. Consequently, all the groups $G_k$ are finite. The
  corresponding Cayley graphs are expanders because $G$ has Kazdhan
  property (T). Since $[x_1,\ {}_k \ x_0] \in G \cap H_k = N_k$, we
  conclude from Corollary \ref{quotgroupest} that
  $$ | G_{3i+j} | \ge [ G : N_{3i+j} ] \ge 2^{8i-1+\mu(j)}. $$ 
\EPf

\begin{rmk}
  Expanders are increasing families of finite graphs with a uniform
  positive lower bound on their {\em combinatorial Cheeger constants}
  (or {\em edge expansion ratios}). The combinatorial Cheeger
  constants for infinite regular tessellations $\G_{p,q}$ of the
  hyperbolic plane (i.e., every vertex is of degree $p$ and every face
  is a $q$-gon) was exactly calculated in \cite{HJL} and
  \cite{HiShi}. Lower bounds for more general planar tessellations (in
  terms of combinatorial curvature) were derived in \cite{KP}. It
  would be interesting to derive similar results for Euclidean and
  hyperbolic buildings and more general non-planar simplicial
  complexes. Note, however, that there is no simple relation between
  the Cheeger constants of infinite graphs and their finite quotients.
\end{rmk}

\section{Further properties of the group $G$ and its pro-$2$
  completion}
\label{outlook}

Let us now have a closer look at the pro-$2$ completion $\widehat G_2$
of our group $G$. Since $N_i = G \cap H_i$ are finite index normal
subgroups of $G$ with $[G : N_i ]$ equals powers of $2$ and $\cap_i
N_i = \{ e \}$, $G$ can be considered as a dense subgroup of $\widehat
G_2$. Moreover, by \cite[Lemma 2.1]{Lub1}, $\widehat G_2$ has a
minimal pro-$2$ presentation given by $\langle x_0,x_1 | r_1,r_2,r_3
\rangle$, with the relations $r_1,r_2,r_3$ defined in
\eqref{r1r2r3}. Consequently, every minimal presentation $\langle X |
R \rangle$ of $\widehat G_2$ satisfies the {\em Golod-Shafarevich
  inequality}
$$ |R| \ge \frac{|X|^2}{4} = 1, $$
since $\widehat G_2$ is not free. (Golod-Shafarevich theorem implies
that every presentation of a finite $p$-group with minimal number of
generators satisfies this inequality.) Further properties of the group
$\widehat G_2$ are given in the following theorem.

\begin{thm}
  The pro-$2$ completion $\widehat G_2$ of $G$ satisfies the
  Golod-Shafarevich inequality even though it is infinite. $\widehat
  G_2$ is not $2$-adic analytic and doesn't contain a free pro-$2$
  subgroup, but it does contain a free subgroup of rank two.
\end{thm}

\begin{proof}
  The statement concerning the Golod-Shafarevich inequality was
  already discussed before. Note that $\widehat G_2 \cap H_i$ is an
  infinite sequence of subgroups of finite index. Next we show that
  $\widehat G_2$ is not $2$-adic analytic: We conclude from
  \eqref{Mipowers} that
  $$ {\widehat G_2}^{\, 2^k} \subset \widehat G_2 \cap H_{2^k-1}. $$
  Corollary \ref{quotgroupest} implies that $[G : N_n] \ge
  2^{2n}$. Consequently, we have
  $$ 
  [\widehat G_2 : \overline{{\widehat G_2}^{\, 2^k}}] \ge [G :
  N_{2^k-1}] \ge 2^{(2^k)}.
  $$  
  By \cite[Thm 3.16]{DdSMS}, $\widehat G_2$ cannot be of finite rank and
  therefore not $2$-adic analytic. 

  The presentation of $\Gamma$ given \eqref{G2} satisfies the
  conditions $C(3)$ and $T(6)$. (In fact, $\Gamma$ is isomorphic to
  the group $G_3 = \langle x | r_3 \rangle$ in \cite[Ex. 3.3]{EH}.)
  Thus we conclude with \cite{EH} that $\Gamma$ (and, therefore, also
  $G$ and $\widehat G_2$) contains a free subgroup of rank two. On the
  other hand, $\widehat G_2$ cannot contain a free pro-$2$ subgroup
  since it is a linear group over a local field (see \cite{BaL}).
\end{proof}

Finally, let us state our conjectures which are based on
MAGMA-computer calculations.

\begin{conj} \label{conjgamma}
  Let $\lambda_i(G)$ and $\gamma_i(G)$ denote the groups in the lower
  exponent-$2$ series and the lower central series of $G$. Then we
  have
  $$ \lambda_i(G) = G \cap H_i \ \text{for $i \ge 1$},$$
  and
  $$ \lambda_i(G) / \lambda_{i+1}(G) \cong \gamma_i(G) / \gamma_{i+1}(G) \
  \text{for $i \ge 2$}.$$
\end{conj}

If Conjecture \ref{conjgamma} is true then Theorem \ref{gammabases} is
still valid if we replace $\lambda_i(G)$ and $\lambda_i(G) \cap
H_{i+1}$ by $\gamma_i(G)$ and $\gamma_{i+1}(G)$, respectively, and,
consequently, the group $\Gamma$ is of finite width $3$ and of finite
average width $(3+3+2)/3 = 8/3$. Moreover, the covering indices of our
tower of expander graphs $\G_i$ are given by the periodic sequence
$4,8,\, \overline{4,8,8}$.

Computer calculations suggest that not only the group $\Gamma$ is of
finite width $3$, but also all groups $\Gamma_\T$ introduced in
\cite[Section 4]{CMSZ} and associated to prime powers $q = p^k$ with
primes $p \neq 3$ (we exclude $p=3$ to avoid torsion phenomena). Here
we expect the following statements to be true:

\begin{conj} \label{conjgt} Let $\Gamma = \Gamma_\T$ be one of the
  groups introduced in \cite[Section 4]{CMSZ}, associated to a prime
  power $q = p^3$ with $p \neq 3$. Then we have the following
  $3$-periodicity for the ranks of the abelian quotients
  $\gamma_i(\Gamma) / \gamma_{i+1}(\Gamma)$ of the lower central
  series for $i \ge 2$:
  $$
  \log_p [ \gamma_i(\Gamma) : \gamma_{i+1}(\Gamma) ] = \begin{cases}
    3, & \text{if
      $i \equiv 0,1 \mod 3$,} \\
    2, & \text{if $i \equiv 2 \mod 3$.} \end{cases}
  $$
\end{conj}

\section{Appendix}

This section is devoted to the proof of Proposition \ref{abc}. For any
$3 \times 3$ matrix $\alpha \in M(3,\F_2)$ and $m,n \in \Z$, we denote
by $E_{m,n}(\alpha)$ the infinite matrix, built up by $3 \times 3$
matrices, which vanishes everywhere expect for its $3 \times 3$ entry
at position $(m,n)$, which coincides with $\alpha$. ($m$ denotes the
$3 \times \infty$ row and $n$ denotes the $\infty \times 3$ column.)
Moreover, given an infinite matrix $A$, built up by $3 \times 3$
matrices, let $\pi_{m,n}(A)$ denote its $3 \times 3$ entry at position
$(m,n)$. Obviously, we have $\pi_{m,n}(E_{m,n}(\alpha)) = \alpha$. For
simplicity we sometimes denote $\pi_{m,n}(A)$ also by $A_{m,n}$.

\begin{lemma} \label{emn}
Let $m,n \ge 1$, $\alpha \in M(3,\F_2)$ and $b \in \Ss$. Then we have
\begin{eqnarray*}
  \pi_{m,n}( M_0(b,\dots)^{-1} E_{m,n}(\alpha) M_0(b,\dots) ) &=& \alpha, \\
  \pi_{m-1,n}( M_0(b,\dots)^{-1} E_{m,n}(\alpha) M_0(b,\dots) ) &=&
  - b(m-1) \cdot \alpha, \\
  \pi_{m,n+1}( M_0(b,\dots)^{-1} E_{m,n}(\alpha) M_0(b,\dots) ) &=&
  \alpha \cdot b(n),
\end{eqnarray*}
where we have taken $m$ and $n$ mod $3$ at the right hand side.
Moreover, we have at all positions $(m',n')$ with $m' > m$ or $n' < n$
$$ \pi_{m',n'}( M_0(b,\dots)^{-1} E_{m,n}(\alpha) M_0(b,\dots) ) = 0. $$
\end{lemma}

\begin{Proof}
  A straightforward calculation shows $M_0(b,\dots)^{-1} =
  M_0(-b,\dots)$. Then we have
  $$ \pi_{m',n'}(A B C) = \sum_{i,j} A_{m',i} B_{i,j} C_{j,n'}, $$
  and in particular
  $$ \pi_{m',n'}(A E_{m,n}(\alpha) C) = A_{m',m} \alpha C_{n,n'}. $$
  The lemma follows now immediately from
  $$\pi_{m-1,m}(M_0(b,\dots)^{-1}) = -b(m-1), \quad
  \pi_{n,n+1}(M_0(b,\dots)) = b(n+1),$$
  and $\pi_{i,j}(M_0(b,\dots)^{\pm 1}) = 0$ for $j < i$.
\end{Proof}

\begin{cor} \label{M0k0}
  We have
  $$
  M_0(b,\dots)^{-1} M_k(a_1,a_2,\dots) M_0(b,\dots) = M_k(a_1,c_2,\dots)
  $$
  with $c_2(i) = a_2(i) - b(i)a_1(i+1) + a_1(i) b(i+k+1)$, where the
  indices $i,i+1,i+k+1$ are taken mod $3$.
\end{cor}

\begin{Proof}
  Note that
  $$ M_k(a_1,a_2,\dots) = I + \sum_{m=1}^\infty \left( \sum_{l=1}^\infty
    E_{m,m+k+l}(a_l(m))\right), $$
  and consequently,
  $$ M_0(b,\dots)^{-1} M_k(a_1,a_2,\dots) M_0(b,\dots) = I + C, $$
  with
  $$
  C = \sum_{m,l=1}^\infty M_0(b,\dots)^{-1} E_{m,m+k+l}(a_l(m)) M_0(b,\dots).
  $$
  Lemma \ref{emn} implies that $I + C$ is of the type
  $M_k(c_1,c_2,\dots)$. Applying Lemma \ref{emn} again, we obtain
  the desired results for the entries $c_1(m)$ and $c_2(m)$ at
  the positions $(m,m+1)$ and $(m,m+2)$.
\end{Proof}

\medskip

\noindent
{\bf Proof of Proposition \ref{abc}:}
We distinguish the cases $k=0$ and $k \ge 1$:

\smallskip

\noindent
{\bf Case $k=0$:} One easily checks that
$$ M_0(a_1,a_2,\dots)^{-1} = M_0(-a_1,d_2,\dots) $$
with $d_2(i) = a_1(i) a_1(i+1) - a_2(i)$ and, using Corollary \ref{M0k0},
\begin{eqnarray*}
[M_0(a_1,a_2,\dots),M_0(b,\dots)] &=& M_0(-a_1,d_2,\dots)M_0(a_1,c_2,\dots)\\
&=& M_0(0,e_2,\dots) = M_1(e_2,\dots)
\end{eqnarray*}
with $e_2(i) = d_2(i) - a_1(i)a_1(i+1) + c_2(i)$. This yields
$$ e_2(i) = c_2(i) = a_2(i) = a_1(i)b(i+1) - b(i) a_1(i+1), $$
finishing this case.

\smallskip

\noindent
{\bf Case $k \ge 1$:} Now we have
$$ M_k(a_1,a_2,\dots)^{-1} = M_k(-a_1,-a_2,\dots) $$
and, using again Corollary \ref{M0k0},
\begin{eqnarray*}
[M_k(a_1,a_2,\dots),M_0(b,\dots)] &=& M_k(-a_1,-a_2,\dots)M_k(a_1,c_2,\dots)\\
&=& M_k(0,c_2-a_2,\dots) = M_{k+1}(c_2-a_2,\dots),
\end{eqnarray*}
where
$$ c_2(i) - a_2(i) = a_1(i) b(i+k+1) - b(i) a_1(i+1). $$
This settles the second case. {\hfill$\square$}

\end{document}